\documentclass[letterpaper, 11pt, nosumlimits, twoside] {amsart}
\title{Rationally connected foliations on surfaces}
\author{Sebastian Neumann}
\email{sebastian.neumann@math.uni-freiburg.de}
\date{18.11.2008}

\usepackage{amsmath}
\usepackage{mathrsfs}  
\usepackage{stmaryrd}  
\usepackage{amssymb}
\usepackage{amsfonts}
\usepackage{amsthm}
\usepackage[latin1]{inputenc}
\usepackage{graphicx}
\usepackage[T1]{fontenc} 
\usepackage{listings}
\usepackage[all]{xy}
\makeindex

\numberwithin{subsection}{section}

\newtheoremstyle{note}
  {}
  {}
  {}
  {}
  {\bfseries}
  {.}
  { }
  {}

\newtheoremstyle{notes}
  {}
  {}
  {\itshape}
  {}
  {\bfseries}
  {.}
  { }
  {}

\newtheoremstyle{numberlessthm}{}{}{\itshape}{}{\bfseries}{.}{ }{\thmnote{#3}}
\newtheoremstyle{numberlessnote}{}{}{}{}{\bfseries}{.}{ }{\thmnote{#3}}
\newtheoremstyle{missingthm}{}{}{}{}{\bfseries}{.}{ }{}

\theoremstyle{notes}
\newtheorem{theorem}{theorem}[section]

\newtheorem{lemma}[theorem]{Lemma}
\newtheorem{thm}[theorem]{Theorem}
\newtheorem{cor}[theorem]{Corollary}

\theoremstyle{note}
\newtheorem{defn}[theorem]{Definition}

\newtheorem{ex}[theorem]{Example}
\newtheorem{rem}[theorem]{Remark}
\newtheorem{notation}[theorem]{Notation}

\theoremstyle{note}
\theoremstyle{numberlessthm}

\theoremstyle{numberlessnote}

\theoremstyle{missingthm}

\theoremstyle{remark}

\newcommand{\R}{\mathbb{R}}
\newcommand{\N}{\mathbb{N}}
\newcommand{\Q}{\mathbb{Q}}

\newcommand{\pro}{{\mathbb P}} 

\newcommand{\cO}{\mathcal{O}}  
\newcommand{\cF}{\mathcal{F}}
\newcommand{\cG}{\mathcal{G}}

\newcommand{\T}{T{\!}X} 

\newcommand{\Ne}{N^1}
\newcommand{\Ner}{N^1_{\mathbb{R}}}

\begin{document}

\maketitle

\begin{abstract} In this short note we study foliations with rationally connected leaves on surfaces. Our main result is that on surfaces there exists a polarisation such that the Harder-Narasimhan filtration of the tangent bundle with respect to this polarisation yields the maximal rationally quotient of the surface.
\end{abstract}

\section{Introduction} Let $X$ be a smooth projective variety. In this note, we are interested in foliations with rationally connected leaves. In \cite{KCT} the authors show how to construct such foliations from the \textit{Harder-Narasimhan Filtration} of the tangent bundle of the variety. This construction heavily depends on a chosen polarisation, and therefore the question arises how this foliation varies changing the polarisation.\newline
There is another way to construct a fibration with rationally connected fibers, the \emph{maximal rationally connected quotient}. This is a rational map, such that the fibers are rationally connected. Almost every rational curve lies in a fiber of this map.\newline
We can ask if the Harder-Narasimhan filtration of the tangent bundle always induces the maximal rationally quotient with respect to any polarisation. The answer is negative already on surfaces due to an example of T. Eckl (\cite{Eckl}).\newline
In this note we will prove that on surfaces there exists always a polarisation, such that the Harder-Narasimhan filtration yields the maximal rationally quotient.
\subsection{Acknowledgements} The author was supported in part by the Gra\-duier\-ten\-kolleg ``Globale Strukturen in Geometrie und Analysis'' of the\linebreak Deutsche For\-schungsgemeinschaft. The result is part of the author's forthcoming Ph.D. thesis written under the supervision of Stefan Kebekus. He would like to thank Stefan Kebekus, Thomas Eckl and Sammy Barkowski for numerous discussions.

The results in this note were presented at a Workshop in Grenoble in April 2008. Seemingly similar results have been obtained independently in \cite{CT}.

\section{Preliminary results and Notation}
Let $X$ be an $n$-dimensional projective variety over the complex numbers with an ample line bundle $H$. Given a torsion-free coherent sheaf $\cF$ on $X$, we define the \emph{slope of} $\cF$ \emph{with respect to} $H$ by

$$\mu_H(\cF):=\frac{c_1(\cF).H^{n-1}}{\textup{rk}(\cF)}.$$

We call $\cF$ \emph{semistable with respect to H} if for any nonzero proper subsheaf $\cG$ of $\cF$ we have $\mu_H(\cG)\leq\mu_H(\cF)$.\newline
If there exists a nonzero subsheaf $\cG\subset\cF$ such that $\mu_H(\cG)>\mu_H(\cF)$, we will call $\cG$ a \emph{destabilizing subsheaf} of $\cF$.\newline

For a proof of the following result, we refer the reader to \cite[Theorem 1.3.4]{HL}.
\begin{thm} Let $\cF$ be a torsion-free coherent sheaf on a smooth projective variety and $H$ be an ample line bundle on $X$. There exists a unique filtration of $\cF$, the so called \textup{Harder-Narasimhan filtration} or \textup{HN-filtration}, depending on the chosen ample line bundle
  $$0=\cF_0\subset\cF_1\subset\ldots\subset\cF_k=\cF$$
with the following properties:
\begin{itemize}
\item[(i)] The quotients $\cG_i:=\cF_i/\cF_{i-1}$ are torsion-free and semistable.
\item[(ii)] The slopes of quotients satisfy $\mu_H(\cG_1)>\ldots>\mu_H(\cG_k)$.
\item[(iii)]The sheaves $\cF_i$ are saturated in $\cF$.
\end{itemize}
\end{thm}

\begin{defn} Let $\cF$ be a torsion-free coherent sheaf on a smooth projective variety. The unique sheaf $\cF_1$ appearing in the Harder-Narasimhan filtration of $\cF$ is called \emph{the maximal destabilizing subsheaf of} $\cF$.
\end{defn}
\begin{defn} Let $\cF$ be a coherent torsion-free sheaf on a smooth projective variety with HNF with respect to an ample line bundle $H$
  $$0=\cF_0\subset\ldots\subset\cF_k=\cF.$$
If the slope of the quotient $\cF_i/\cF_{i-1}$ is positive with respect to $H$, then  $\cF_i$ is called \emph{positive with respect to} $H$.
\end{defn}
\begin{rem}
Note that the construction of the HNF naturally extends to $\Q$-- and $\R$--divisors, i.e we do not need to assume the chosen polarisation to be integral.
\end{rem}
Obviously, the Harder-Narasimhan filtration depends only on the numerical class of the chosen ample bundle. In particular it makes sense to ask how the filtration of a given sheaf depends on the ample bundle sitting in the finite dimensional  vectorspace of all divisors modulo numerical equivalence.
\begin{notation} We sometimes omit the polarisation in the notation of the slope, that is we write $\mu(\cF)$ for the slope of a sheaf $\cF$ with respect to a polarisation.
\end{notation}
We can now  state an important result originally formulated by  Miyaoka and explicitely shown in \cite{KCT}. For a survey on these and related results we refer the reader to \cite{KC}.
\begin{thm}\cite[Theorem 1]{KCT}\label{Keb} Let $X$ be a smooth projective variety and let $$0=\cF_0\subset \cF_1\subset\ldots\subset\cF_k=\T$$ be the Harder-Narsimhan filtration of the tangent bundle with respect to a polarisation $H$. Write $\mu_i:=\mu_H(\cF_i/\cF_{i-1})$ for the slopes of the quotients. Assume $\mu_1>0$ and set $m:=\max\left\{i\in\N|\mu_i>0\right\}$.
Then each $\cF_i$ with $i\leq m$ is a foliation with algebraic leaves and for general $x\in X$ the closure of the leaf through $x$ is rationally connected.
\end{thm}

Let $X$ be a smooth projective variety and assume the conditions of Theorem (\ref{Keb}) are fulfilled. Thus  we obtain foliations $\cF_1,\ldots,\cF_k$ with algebraic and rationally connected leaves. By setting
	$$\begin{array}{crcl}
	q_i:&X&\dashrightarrow &\textup{Im}(q_i)\subset\textup{Chow}(X)\\
	&x &\mapsto &\cF_i\textup{-leaf through } x 
	\end{array}$$
we obtain a map, such that the closure of the general fibre is rationally connected, see \cite{KCT}, section 7.\newline
There is another map with this property called the \emph{maximal rationally connected quotient} or \emph{MRC-quotient} for short based on a construction by Campana \cite{Campana1}, \cite{Campana2} and Koll{\'a}r-Miyaoka-Mori \cite{KMM}.
\begin{thm}\cite[Chapter IV, Theorem 5.2]{K} Let $X$ be a smooth projective variety. There exists a variety $Z$ and a rational map $\phi:X\dashrightarrow Z$ with the following properties:
\begin{itemize}
  \item the fibers of $\phi$ are rationally connected,
  \item a very general fiber of $\phi$ is an equivalence class under rationally connectivity and
  \item up to birational equivalence the map $\phi$ and the variety $Z$ are unique.
\end{itemize}
\end{thm}
It is possible to ask if the HNF with respect to a certain polarisation yields the MRC-quotient. We will give a positive answer in the next section.

\section{Rationally connected fibration on surfaces and the MRC-quotient}\label{ergebnis}
In this section $X$ will be a smooth projective surface over the complex numbers.
We want to investigate the regions in the ample cone which induce the same HN-filtration. More precisely we divide the ample cone into parts, such that in each part, we get the same HN-filtration of the tangent bundle.  With this at hand we are able to show that the MRC-quotient comes from the Harder-Narasimhan filtration with respect to a certain polarisation.
%
%
%
%

In order to compute the HN-filtration of the tangent bundle on surfaces, we only have to search for a destabilizing subbundle, such that the quotient is torsion-free.
\begin{lemma}\label{onebundle} Let $X$ be a smooth projective surface. If $\cF\subset \T$ is a destabilizing subsheaf with respect to a polarisation such that $\T/\cF$ is torsion-free, then the Harder-Narasimhan filtration is given by $0\subset \cF\subset \T$.
\end{lemma}
\begin{proof}
Consider the exact sequence
	$$0\rightarrow \cF \rightarrow \T\rightarrow \T/\cF\rightarrow 0.$$
Using that the rank and the chern class is additiv in short exact sequences, we obtain
$$\mu (\T)=\frac{1}{2}\mu(\T/\cF)+\frac{1}{2}\mu(\cF).$$ Since $\mu(\cF)>\mu(\T)$, we therefore have $\mu(\cF)>\mu(\T/\cF)$. That is, $$0\subset \cF\subset \T$$ satisfies the properties of the Harder-Narasimhan filtration and by the uniqueness of the HN-filtration, we finish the proof.
\end{proof}
\begin{notation} 
We will write $\Ne(X)$ for the N\'eron-Severi group and $\textup{N}_{\mathbb{Q}}^1(X)$ (resp. $\Ner (X)$) for the vectorspace of $\Q$--divisors (resp. $\R$--divisors) modulo numerical equivalence on $X$. The convex cone of all ample $\R$--divisors in $\Ner(X)$ will be denoted by $\textnormal{Amp}_{\mathbb{R}}(X)$.
\end{notation}
Now we denote the regions in $\textnormal{Amp}_{\mathbb{R}}(X)$, we are interested in. Let $H\in \Ner(X)$ be an ample bundle. If $\T$ is not semistable with respect to $H$, let $D_{\T}(H)$ be the unique maximally destabilizing subbundle of $\T$. In this case, we call $$\Delta_{H}:=\left\{\tilde{H} \in \textnormal{Amp}_{\mathbb{R}}(X)| D_{\T}(H)=D_ {\T}(\tilde{H})\right\}$$ the \textit{destabilisation chamber with respect to $H$}. Note that if the tangent bundle is semistable with respect to a certain polarisation, then we get a chamber, such that for all polarisations in this chamber $\T$ is semistable. This region will be denoted by $\Sigma$ and we call it the \textit{semistable chamber}. That is $$\Sigma:=\left\{H \in \textnormal{Amp}_{\mathbb{R}}(X)| \T \textnormal{ is semistable with respect to } H\right\}.$$  

The destabilizing chambers and the chamber of stability give a decomposition of the ample cone. Concerning the structure of these chambers we prove the following lemma.
\begin{lemma}\label{structureOfChambers} Let $X$ be a smooth projective surface. We have
	\begin{enumerate}
		\item[(i)] The destabilisation chambers and the semistable chamber are convex cones in $\textnormal{Amp}_{\mathbb{R}}(X)$.
		\item[(ii)] The semistable chamber is closed in $\textnormal{Amp}_{\mathbb{R}}(X)$.
		\item[(iii)] The destabilisation chambers are open in $\textnormal{Amp}_{\mathbb{R}}(X)$.
	\end{enumerate}
\end{lemma}
\begin{proof} The convexity property follows from the linearity of the intersection product and the uniqueness of the HN-filtration. In detail: Let $H_{1}$ and $H_{2}$ be two polarisations belonging to the same chamber. This means, that there exists a subsheaf $L\subset \T$ which is maximal destabilizing both with respect to  $H_1$ and $H_2$. We obtain
$$c_1(L).(H_{1}+H_{2})=c_1(L).H_{1}+c_1(L).H_{2}>\frac{1}{2}c_{1}(\T).(H_{1}+H_{2}).$$
By Lemma (\ref{onebundle}), we deduce that $L$ is the maximally destabilizing subsheaf with respect to $H_{1}+H_{2}$.\newline
In the same spirit we prove the convexity of $\Sigma$: Let $H_{1}$ and $H_{2} \in \Sigma$. This means by definition that for all subsheaves $L\subset \T$ the inequality $c_1(L).H_{i}\leq \frac{1}{2}c_{1}(\T).H_{i}$, $i=1,2$ holds. Therefore by linearity of the intersection product the inequality holds for $H_{1}+H_{2}$. This proves $(i)$.\newline
The second statement follows from the continuity of the intersection product. Let $H_{n}\in\Sigma$ be a sequence  converging to $H\in\textnormal{Amp}(X)$. Then for a given subsheaf $L\subset\T$ one has $\mu_{H_{n}}(L)\leq \mu_{H_{n}}(\T)$ for each $n\in\N$. Passing to the limit the inequality still holds by continuity.\newline
Statement $(iii)$ is also a direct consequence of the continuity of the intersection product: Let $H\in\Delta_{\tilde{H}}$. Let $L$ be the maximal destabilizing subsheaf of the tangent bundle. That is we have $$(c_1(L)-\frac{1}{2}c_{1}(\T)).H>0.$$ This inequality holds in a neighbourhood of $H$. By Lemma (\ref{onebundle}) the subsheaf $L$ gives the HN-filtration in a neighbourhood of $H$.
\end{proof}
%
%
In order to prove our main result, we use the following corollary.
\begin{cor}\label{corollary} Let $X$ be a smooth projective surface. Let $l$ be a linesegment in $\textup{Amp}_{\R}(X)$. If $\T$ is not semistable with respect to all polarisation on $l$, then $l$ lies completely in one destabilizing chamber.
\end{cor}

In order to prove semistability of the tangent bundle on certain surfaces having many automorphisms, we give a useful lemma.
\begin{lemma}\label{FolAndAut} Let $X$ be a smooth projective surface and let $\sigma\in\textnormal{Aut}^{0}(X)$. Let $\cF$ be the maximal destabilizing subsheaf of $\T$ with respect to a polarisation. Then we have $\sigma^{*}(c_1(\mathcal{F}))=c_1(\mathcal{F})$. In particular: If the slope of $\cF$ is positive then the automorphism $\sigma$ maps each leaf of $\mathcal{F}$ to another leaf of $\mathcal{F}$.
\end{lemma}
\begin{proof}Let $H\in\Ne(X)$ and let $\cF$ be the maximal destabilizing subsheaf of $\T$ with respect to $H$.  We compute the slope of $\sigma^{*}(T_{\mathcal{F}})$:
$$
\begin{array}[t]{rcl}
	\mu(\sigma^{*}(c_1(\cF))) &=& H.\sigma^{*}(c_1(\cF))\\
	 & = &  \sigma_{*}(H).c_1(\cF)\\
	 & = &  H.c_1(\cF)\\
	 & > &\frac{1}{2}c_{1}(\T).H.
\end{array}
$$
We give an explanation of the third equality: Recall that the group of automorphisms acts on the N\'eron-Severi group. Since $N^1(X)$ is discrete $\textup{Aut}^0(X)$ acts trivially on $N^1(X)$, i.e.  the pushforward of $\sigma$ has to give the same element in $N^{1}(X)$.\newline
Therefore, we have shown that $\sigma^{*}(T_{\mathcal{F}})$ is also a destabilizing bundle and by the uniqueness of the HN-filtration we finish the proof.
\end{proof}
\begin{ex}\textit{Hirzebruch Surfaces}\newline
Let $\Sigma_{n}$ be the $n$-th Hirzebruch surface, $\pi: \Sigma_{n}\longrightarrow \mathbb{P}^{1}$ the projection onto the projective line. We denote the fiber under the projection with $f$ and the distinguished section with selfintersection $-n$ with $C_{0}$. Recall (see \cite{H}, chapter V.2)  that $\textnormal{Num}_{\mathbb{R}}(\Sigma_{n})=<C_{0},f>$ and a divisor $D\equiv_{num}aC_{0}+bf$ is ample if and only if $a>0$ and $b>an$. The canonical bundle is given by $-K_{\Sigma_{n}}=2C_{0}+(2+n)f$.
The relative tangent bundle of $\pi$ is a natural candidate for a destabilizing subbundle. We have the sequence
\begin{center}
	$0\rightarrow T_{\Sigma_{n}/\mathbb{P}^{1}}\rightarrow T\Sigma_{n} \rightarrow \pi^{*}T\mathbb{P}^{1} \rightarrow 0$
\end{center}
Let $H:=xC_{0}+yf$ be a polarisation. Then one can compute that $T_{\Sigma_{n}/\mathbb{P}^{1}}$ is destabilizing if and only if $-2x-nx+2y>0$. In particular we compute for $n\geq 2$:
\begin{center}
	$-2x-nx+2y>-2x-nx+2nx=-2x+nx\geq0$.
\end{center}
Therefore, for $n\geq2$ the HN-filtration is given by
\begin{center} 
	$0\subset T_{X/\mathbb{P}^{1}}\subset \T$
\end{center}
for all polarisations. In other words we obtain only one destabilizing chamber.

\begin{figure}[htbp]
\includegraphics[scale=0.8]{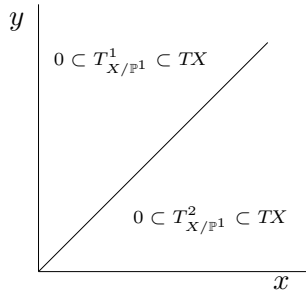}
\caption{The ample cone of $X=\Sigma_{0}$ and the chamber structure. Here $T^{1}_{X/\mathbb{P}^{1}}$ and $T^{2}_{X/\mathbb{P}^{1}}$ denote the relative tangent bundle of the first and second projection.}
\begin{picture}(0,0)
\put(-45.0,140.0){\tiny{$0\subset T^{1}_{X/\mathbb{P}^{1}}\subset \T$}}
\put(-15.0,80.0){\tiny{$0\subset T^{2}_{X/\mathbb{P}^{1}}\subset \T$}}
\put(-62,155.0){$y$}
\put(38.0,55.0){$x$}
\end{picture}
\end{figure}

For $n=0$ we have $\Sigma_{0}=\mathbb{P}^{1}\times\mathbb{P}^{1}$ and we get three chambers. Two destabilizing chambers correspond to the two relative tangent bundles of the projections. They are cut out by the inequalitys $x>y$ and $x<y$.  There is a chamber of semistability, which is determined by the equation $x=y$.

For $n=1$ we see that for $x>\frac{3}{2}y$ the relative tangent bundle is destabilizing. Since $\Sigma_{1}$ is the projective plane with one point blown up, say $p$, we know the group of automorphisms, which is the automorphism group of the projective plane leaving $p$ fixed. The destabilizing foliation corresponds to the radial foliation through $p$ in the plane.\newline
So if there were another foliation $\cF$ coming from the Harder-Narasimhan filtration, we could deform the leaves with these automorphisms. Then we would again obtain leaves of this foliation by Lemma (\ref{FolAndAut}). So unless $\cF$ is the radial foliation, we could produce infintely many singularities, which is absurd. Therefore we conclude that the tangent bundle is semistable for $x\leq\frac{3}{2}y$.
\begin{rem} From this example follows, that the stabilizing chamber can be both lower dimensional or equal the dimension of the ample cone.
\end{rem}

\end{ex}
Now we want to answer the question if there always exists a polarisation, such that the HNF gives rise to the MRC-quotient.
\begin{thm} Let $X$ be a uniruled  projective surface. Then there exists a polarisation, such that the maximal rationally quotient of $X$ is given by the foliation associated to highest positive term in the HNF associated with this polarisation.
\end{thm}
\begin{proof} 
To start, observe that there is always a polarisation $H_{1}$ such that $c_{1}(\T).H_{1}>0$. Assume on the contrary, that each ample bundle intersects $c_{1}(\T)$ negatively or zero. This means that $K_{X}$ is pseudoeffective. On the other hand there exists a free rational curve $f:\pro^1\rightarrow X$. See \cite[Corollary 4.11]{Deb} for a prove of the existence of such a curve. Writing $$f^*(\T)=\cO(a_1)\oplus\cO(a_2)$$  with $a_1+a_2\geq 2$, we compute $$K_X.f_*\pro^1=-a_1-a_2\leq -2,$$ a contradiction.\newline
First let us assume that $X$ is not rationally connected. Then we take the polarisation $H_{1}$ with $c_{1}(\T).H_{1}$>0. There exists a destabilizing subsheaf $\cF$ of $\T$, since otherwise $X$ would be rationally connected by Theorem (\ref{Keb}). Furthermore the slope of $\cF$ has to be bigger than $c_1(\T).H_1$ and therefore positive. So this sheaf will give a foliation with rationally connected leaves and hence the maximal rationally quotient.\newline
Consider the case where $X$ is rationally connected. Then we fix a very free rational curve $l$ on $X$. For a proof of the existence of a very free rational curve see \cite[Corollary 4.17]{Deb}. This means that $\T|_l$ is ample. So we know that each quotient of $\T|_l$ has strictly positive degree.\newline
Since $l$ is movable, it is in particular nef. Let $H_{2}$ be an ample class. Because $l$ is nef, we know that $H_{\epsilon}:=l+\epsilon H_{2}$ is ample in $N^{1}_{\mathbb{Q}}(X)$ for any $\epsilon >0$. Observe that $c_{1}(\T).H_{\epsilon}>0$ for sufficiently small $\epsilon$, say for $0<\epsilon<1$. If $\T$ is semistable with respect to a certain polarisation $H_{\epsilon}$ with $0<\epsilon<1$, the claim follows since $\T$ has positive slope and induces a trivial foliation which gives the rationally connected  quotient. If $\T$ is not semistable for all polarisations $H_{\epsilon}$ with $0<\epsilon<1$, let $L_{\epsilon}$ be the destabilizing subsheaf of $\T$ with respect to $H_{\epsilon}$. Because of Corollary (\ref{corollary}) the ray $H_{\epsilon}$ stays in one destablizing chamber. So $L:=L_{\epsilon}$ remains constant.\newline
Now it is clear that for sufficiently small $\epsilon$ both, the slope of $L$ and the slope of $\T/L$ will be positive with respect to $H_{\epsilon}$. Therefore the HN-fitration of $\T$ with respect to $H_{\epsilon}$ yields the maximal rationally connected quotient.
\end{proof}

\bibliographystyle{alpha}
\bibliography{books}

\def\cprime{$'$}

\end{document}